\documentclass[12pt,english]{article}
\usepackage{mathptmx}

\usepackage[T1]{fontenc}
\usepackage[latin9]{inputenc}
\usepackage[letterpaper]{geometry}
\usepackage{float}
\usepackage{amsthm}
\usepackage{amstext}
\usepackage{graphicx}
\usepackage{amssymb}

\makeatletter
  \theoremstyle{plain}
  \newtheorem*{thm*}{Theorem}
\theoremstyle{plain}
\newtheorem{thm}{Theorem}
  \theoremstyle{definition}
  \newtheorem*{example*}{Example}
  \theoremstyle{plain}
  \newtheorem{lem}[thm]{Lemma}
  \theoremstyle{remark}
  \newtheorem*{rem*}{Remark}
  \theoremstyle{plain}
  \newtheorem{cor}[thm]{Corollary}
  \theoremstyle{remark}
  \newtheorem*{acknowledgement*}{Acknowledgement}

\@ifundefined{showcaptionsetup}{}{%
 \PassOptionsToPackage{caption=false}{subfig}}
\usepackage{subfig}
\makeatother

\usepackage{babel}

\begin{document}

\title{Bounded Mean Oscillation and Bandlimited Interpolation in the Presence
of Noise}

\author{Gaurav Thakur%
\thanks{Program in Applied and Computational Mathematics, Princeton University,
Princeton, NJ 08544, USA, email: gthakur@princeton.edu%
}}

\date{September 3, 2010}
\maketitle
\begin{abstract}
We study some problems related to the effect of bounded, additive
sample noise in the bandlimited interpolation given by the Whittaker-Shannon-Kotelnikov
(WSK) sampling formula. We establish a generalized form of the WSK
series that allows us to consider the bandlimited interpolation of
any bounded sequence at the zeros of a sine-type function. The main
result of the paper is that if the samples in this series consist
of independent, uniformly distributed random variables, then the resulting
bandlimited interpolation almost surely has a bounded global average.
In this context, we also explore the related notion of a bandlimited
function with bounded mean oscillation. We prove some properties of
such functions, and in particular, we show that they are either bounded
or have unbounded samples at any positive sampling rate. We also discuss
a few concrete examples of functions that demonstrate these properties.
\end{abstract}
\noindent {\small Mathematics Subject Classification (2010): Primary
30E05; Secondary 94A20, 30D15, 30H35}{\small \par}

\noindent {\small Keywords: Sampling theorem, Nonuniform sampling,
Paley-Wiener spaces, Entire functions of exponential type, BMO, Sine-type
functions}{\small \par}

\section{Introduction\label{SecIntro}}

The classical Whittaker-Shannon-Kotelnikov (WSK) \textit{sampling
theorem} is a central result in signal processing and forms the basis
of analog-to-digital and digital-to-analog conversion in a variety
of contexts involving signal encoding, transmission and detection.
If we normalize the Fourier transform as $\hat{f}(\omega)=\int_{-\infty}^{\infty}f(t)e^{-2\pi i\omega t}dt$,
then the sampling theorem states that a function $f\in L^{2}(\mathbb{R})$
with $\mathrm{supp}(\hat{f})\subset[-\frac{b}{2},\frac{b}{2}]$ can
be expressed as a series of the form\begin{equation}
f(t)=\sum_{k=-\infty}^{\infty}a_{k}\frac{\sin(\pi(bt-k))}{\pi(bt-k)},\label{WSKSeries}\end{equation}

\noindent where $a_{k}=f(k)$ are its samples. Conversely, for a given
collection of data $\{a_{k}\}\in l^{2}$, the series (\ref{WSKSeries})
defines a function in $L^{2}(\mathbb{R})$ with $\mathrm{supp}(\hat{f})\subset[-\frac{b}{2},\frac{b}{2}]$
called the \textit{bandlimited interpolation} of $\{a_{k}\}$. The
calculation or approximation of this series is a standard procedure
in many applications. For example, in audio processing it is used
for resampling signals at a higher rate, typically by applying a lowpass
filter to the piecewise-constant zero order hold function of the samples
\cite{CR83}. In this paper, we consider the situation of bounded
noise in the samples $a_{k}$. Building on recent work by Boche and
Mönich on related problems \cite{BM07,BM09,BM09-2}, we study some
properties of the effect of the noise on the bandlimited interpolation
$f$.\\

\noindent Before we discuss our problems, it will be convenient to
define the \textit{Paley-Wiener spaces} for $1\leq p\leq\infty$ by\[
PW_{b}^{p}=\left\{ f\in L^{p}:\mathrm{supp}(\hat{f})\subset\left[-\frac{b}{2},\frac{b}{2}\right]\right\} ,\]

\noindent where $\hat{f}$ is interpreted in the sense of tempered
distributions. Our notation $PW_{b}^{p}$ essentially follows Seip
\cite{Se04}, and is slightly different from the one used by Boche
and Mönich. Without loss of generality, we will set $b=1$ in what
follows.\\

Returning to the series (\ref{WSKSeries}), we consider corrupted
samples of the form $a_{k}=T_{k}+N_{k}$, where $T_{k}$ are the true
samples and $N_{k}$ is some form of noise, and we correspondingly
write $f(t)=T(t)+N(t)$. One obstacle we face is that the noise $\{N_{k}\}$
may not naturally decay in time alongside the signal, and even if
$\{T_{k}\}\in l^{2}$, it is often more physically meaningful to consider
$\{N_{k}\}\in l^{\infty}$. The WSK sampling theorem shows that for
any collection of samples $\{a_{k}\}\in l^{2}$, there exists a unique
function $f\in PW_{1}^{2}$ with $f(k)=a_{k}$. However, for bounded
samples $\{a_{k}\}\in l^{\infty}$, the series (\ref{WSKSeries})
does not necessarily converge. In fact, a given $\{a_{k}\}\in l^{\infty}$
may correspond to multiple functions $f\in PW_{1}^{\infty}$, or to
no such function \cite{BM07}.\\

A simple example of the former possibility (non-uniqueness) is given
by $a_{k}\equiv0$, which corresponds to the functions $f(t)\equiv0$
and $f(t)=\sin(\pi t)$. It turns out that adding one extra sample
to the collection $\{a_{k}\}$ resolves this ambiguity, and allows
us to consider the unique bandlimited interpolation of any bounded
data $\{a_{k}\}\in l^{\infty}$. We discuss the details of this procedure
in Section 3. The latter possibility (non-existence) is less obvious,
but in \cite{BM07}, Boche and Mönich presented an explicit example
of this phenomenon. They showed that for the samples given by $a_{k}=0$,
$k<1$, and $a_{k}=(-1)^{k}/\log(k+1)$, $k\geq1$, there is no $f\in PW_{1}^{\infty}$
with $f(k)=a_{k}$. It is also possible to construct other, similar
examples using standard special functions, and we describe one such
sequence of $\{a_{k}\}$ in Section \ref{SecBounded} and discuss
its properties.\\

The main observation of this paper is that such examples of $\{a_{k}\}$
are in a sense {}``highly oscillating.'' By assuming that the noise
$N_{k}$ is statistically incoherent and defining $N(t)$ carefully,
we can rule out these examples and obtain sharper statements on the
behavior of $N(t)$. More precisely, we show in Section \ref{SecRandom}
that if $N_{k}$ is a uniformly distributed, independent white noise
process, then $\sup_{r>0}\frac{1}{2r}\int_{-r}^{r}|N(t)|dt<\infty$
almost surely. In other words, the average of $|N(t)|$ is globally
bounded. We find that this result does not generally hold for $\{N_{k}\}\in l^{\infty}$
that lack such a statistical condition, and we discuss examples that
illustrate the differences.\\

We also study a second topic motivated by further understanding $N(t)$.
As discussed in \cite{Eo95}, the WSK series (\ref{WSKSeries}) can
be interpreted as a discrete Hilbert transform operator $H$, mapping
a space of samples into a space of bandlimited functions (see also
\cite{BMS09} and \cite{LS97}). The Plancherel formula shows that
$H$ maps $l^{2}$ into $PW_{1}^{2}$. In fact, $H$ also maps $l^{p}$
into $PW_{1}^{p}$ for any $1<p<\infty$, and the series (\ref{WSKSeries})
converges for any $\{a_{k}\}\in l^{p}$ \cite{Le96}. This can be
compared with the continuous Hilbert transform, and more generally
any Calderon-Zygmund singular integral operator, which maps $L^{p}$
into itself for any $1<p<\infty$. Such operators behave differently
for $p=\infty$, mapping $L^{\infty}$ into the space $BMO$ of functions
with \textit{bounded mean oscillation} \cite{St93}.\\

It is thus reasonable to expect that if we consider samples $\{a_{k}\}\in l^{\infty}$,
the {}``right'' target space for $H$ may be one of bandlimited
functions lying in the space $BMO$. However, this heuristic reasoning
turns out to be incorrect. We consider bandlimited $BMO$ functions
in Section \ref{SecBMO} and establish some of their properties. In
particular, we find that such a function $f$ is either in $L^{\infty}$
or that its samples $\{f(\frac{k}{s})\}$ are unbounded for any sampling
rate $s>0$. We exhibit a concrete example of such a function, and
study it in the context of our other results.\\

We review some existing theory on bandlimited functions and the space
$BMO$ in Section \ref{SecBG}, and discuss some preliminary results
in Section \ref{SecBounded}. The main results of the paper are presented
in Sections \ref{SecRandom} and \ref{SecBMO}. We also develop our
results for a class of general, nonuniformly spaced interpolation
points, given by zeros of \textit{sine-type functions}. The above
discussion for uniformly spaced points is a special case.

\section{Background Material\label{SecBG}}

We will write $f_{1}\lesssim f_{2}$ if the inequality $f_{1}\leq Cf_{2}$
holds for a constant $C$ independent of $f_{1}$ and $f_{2}$. We
define $f_{1}\gtrsim f_{2}$ similarly, and write $f_{1}\eqsim f_{2}$
if both $f_{1}\lesssim f_{2}$ and $f_{1}\gtrsim f_{2}$. For a set
of points $Y=\{y_{k}\}$ and an extra element $\tilde{y}$, we denote
the collection $\{y_{k}\}\bigcup\{\tilde{y}\}$ by $\tilde{Y}$, with
$||\tilde{Y}||_{l^{p}}:=(||Y||_{l^{p}}+|\tilde{y}|^{p})^{1/p}$ and
$||\tilde{Y}||_{l^{\infty}}:=\max(\left\Vert Y\right\Vert _{l^{\infty}},|\tilde{y}|)$.
These conventions will be used throughout the paper.\\

\noindent We first review a basic, alternative formulation of $PW_{b}^{p}$,
$1\leq p\leq\infty$. An entire function $f$ is said to be of \textit{exponential
type $b$} if\[
b=\inf\left(\beta:|f(z)|\leq e^{\beta|z|},z\in\mathbb{C}\right).\]
We denote this by writing $\mathrm{type}(f)=b$, and by $\mathrm{type}(f)=\infty$
if $b=\infty$ or $f$ is not entire. By the Paley-Wiener-Schwartz
theorem \cite{Ho03}, $PW_{b}^{p}$ can be equivalently described
as the space of all entire functions with $\mathrm{type}(f)\leq\pi b$
whose restrictions to $\mathbb{R}$ are in $L^{p}$. It also follows
that $PW_{b}^{p}\subset PW_{b}^{q}$ for $p<q$. Functions $f\in PW_{b}^{p}$
satisfy the classical estimates $\left\Vert f'\right\Vert _{L^{p}}\leq\pi b\left\Vert f\right\Vert _{L^{p}}$
and $\left\Vert f(\cdot+ic)\right\Vert _{L^{p}}\leq e^{\pi b|c|}\left\Vert f\right\Vert _{L^{p}}$,
respectively known as the \textit{Bernstein} and \textit{Plancherel-Polya
inequalities }\cite{Le96,Se04}\textit{.}\\

There is a rich and well-developed theory of nonuniform sampling for
functions in $PW_{b}^{p}$. We only cover a few aspects of it that
we will need in this paper, and refer to \cite{Se04} and \cite{Yo01}
for more details. We consider a sequence of points $X=\{x_{k}\}\subset\mathbb{R}$,
indexed so that $x_{k}<x_{k+1}$. The \textit{separation constant}
of $X$ is defined by $\lambda(X)=\inf_{k}|x_{k+1}-x_{k}|$, and $X$
is said to be \textit{separated} if $\lambda(X)>0$. The \textit{generating
function }of $X$ is given by the product\begin{equation}
S(z)=z^{\delta_{X}}\lim_{r\rightarrow\infty}\prod_{0<|x_{k}|<r}\left(1-\frac{z}{x_{k}}\right),\label{GenFunc}\end{equation}

\noindent where $\delta_{X}=1$ if $0\in X$ and $\delta_{X}=0$ otherwise.
For real and separated $X$, such a function $S$ is said to be \textit{sine-type}
if the following conditions hold:\\

\noindent (I) The product (\ref{GenFunc}) converges and $\mathrm{type}(S)=\pi b<\infty$.

\noindent (II) For any $\epsilon>0$, there are positive constants
$C_{1}(\epsilon)$ and $C_{2}(\epsilon)$ such that whenever $\mathrm{dist}(z,X)>\epsilon$,
\begin{equation}
C_{1}(\epsilon)\leq e^{-\pi b\mathrm{|Im(z)|}}|S(z)|\leq C_{2}(\epsilon).\label{SineType}\end{equation}

\noindent It can be shown that condition (II) is equivalent to requiring
that the bounds (\ref{SineType}) only hold in some half plane $\{z:\mathrm{|Im}(z)|\geq c\}$,
$c>0$. Furthermore, a sine-type function $S$ also satisfies the
bounds $|S'(x_{k})|\eqsim1$ and forces $X$ to satisfy $\sup_{k}|x_{k+1}-x_{k}|<\infty$
\cite{Le96}.\\

\noindent Now suppose the sequence $X=\{x_{k}\}$ has a sine-type
generating function $S$ with $\mathrm{type}(S)=\pi b$. Let $1<p<\infty$.
Then any $f\in PW_{b}^{p}$ can be expressed in terms of its samples
$a_{k}=f(x_{k})$,\begin{equation}
f(z)=\sum_{k=-\infty}^{\infty}a_{k}\frac{S(z)}{S'(x_{k})(z-x_{k})},\label{SampSeries1}\end{equation}
with uniform convergence on compact subsets of $\mathbb{C}$. Conversely,
for any $\{a_{k}\}\in l^{p}$, the series (\ref{SampSeries1}) converges
uniformly on compact subsets of $\mathbb{C}$ and defines a function
$f\in PW_{b}^{p}$ with $a_{k}=f(x_{k})$ \cite{Le96}.\\

The simplest example of a sequence $X$ with a sine-type generating
function is the uniform sequence $x_{k}=\frac{k}{b}$, for which $S(z)=\frac{\sin(\pi bz)}{\pi b}$
and the expansion (\ref{SampSeries1}) reduces to the WSK sampling
theorem. More generally, any finite union of uniform sequences has
a sine-type generating function. As a more interesting example, the
Bessel function $J_{0}$ has real, separated zeros, satisfies $J_{0}(z)=J_{0}(-z)$,
and has the asymptotic formula $J_{0}(z)=\sqrt{\frac{2}{\pi z}}\cos(z-\frac{\pi}{4})(1+O(\frac{1}{z}))$
as $|z|\to\infty$ and $|\arg z|<\pi$ (see \cite{WW27}). This implies
that for sufficiently small $\epsilon>0$, $S(z)=zJ_{0}(\frac{\pi z}{2})J_{0}(\frac{\pi(z+\epsilon)}{2})$
is a sine-type function with $\mathrm{type}(S)=\pi$. Sequences $X$
with sine-type generating functions are not the most general class
for which $f$ has an expansion of the form (\ref{SampSeries1}),
but they have several convenient properties and cover some important
cases encountered in applications, such as that of periodic interpolation
points. Such sequences $X$ and various properties of the series (\ref{SampSeries1})
have recently been studied in \cite{BM09} in a computational context.\\

The above results do not directly carry over to bounded functions
$f\in PW_{b}^{\infty}$, but in this case we still have the following
theorem \cite{Be89}.
\begin{thm*}
\label{ThmPWInf}(Beurling) For a sequence $X=\{x_{k}\}$, let $N(X,I)$
be the number of $x_{k}$ in an interval $I$. Then $\left\Vert f\right\Vert _{L^{\infty}}\eqsim\left\Vert f(X)\right\Vert _{l^{\infty}}$
for all $f\in PW_{b}^{\infty}$ if and only if\[
D^{-}(X):=\limsup_{r\to\infty}\,\inf_{a}\frac{N(X,[a,a+r))}{r}>b.\]

\end{thm*}
\noindent $D^{-}(X)$ is called the \textit{lower uniform density}
of $X$. For a uniform sequence $x_{k}=\frac{k}{s}$, $D^{-}(X)=s$,
and Beurling's theorem implies that $f\in PW_{b}^{\infty}$ is uniquely
determined by its samples if we oversample it beyond its Nyquist rate.\\

We finally review a few properties of the Banach space $BMO$ of functions
with bounded mean oscillation, which has been studied extensively
in connection with singular integral operators. It is defined by\[
\left\{ f:\left\Vert f\right\Vert _{BMO}=\sup_{I}\frac{1}{|I|}\int_{I}\left|f(t)-\frac{1}{|I|}\int_{I}f(s)ds\right|dt<\infty\right\} ,\]

\noindent where the supremum runs over all real intervals $I$. The
quantity $\left\Vert f\right\Vert _{BMO}$ is technically a seminorm,
since $\left\Vert f\right\Vert _{BMO}$=$\left\Vert f+c\right\Vert _{BMO}$
for any constant $c$. Now for any $g\in L^{1}$, we denote its Hilbert
transform by $\mathcal{H}g(z):=\int_{-\infty}^{\infty}\frac{g(t)}{\pi(t-z)}dt$
and its Riesz projections by $\mathcal{P}^{\pm}g:=\left(g\pm i\mathcal{H}g\right)/2$.
We can then consider the {}``real'' Hardy space $H^{1}(\mathbb{R})$,
given by\[
\left\{ f:\left\Vert f\right\Vert _{H^{1}(\mathbb{R})}=\left\Vert f\right\Vert _{L^{1}}+\left\Vert \mathcal{H}f\right\Vert _{L^{1}}<\infty\right\} .\]

\noindent Finally, it will also be useful to define the subspaces\begin{eqnarray*}
\mathbf{U}_{1} & = & \{f\in C_{0}^{\infty}:\int_{-\infty}^{\infty}f(t)dt=0\}\\
\mathbf{U}_{2} & = & \{f\in H^{1}(\mathbb{R}):(1+t^{2})|\mathcal{P}^{+}f(t)|\in L^{\infty}\}\end{eqnarray*}

\noindent which are both norm dense in $H^{1}(\mathbb{R})$ \cite{Ga07,St93}.
These spaces are all closely related, as the following theorem shows.
\begin{thm*}
\noindent \label{ThmBMO1}(Fefferman) $BMO$ is the dual space of
$H^{1}(\mathbb{R})$. More specifically, we have the inequality\[
\left\Vert f\right\Vert _{BMO}\eqsim\sup_{g\in\mathbf{U}}\frac{1}{\left\Vert g\right\Vert _{H^{1}(\mathbb{R})}}\left|\int_{-\infty}^{\infty}f(t)\overline{g(t)}dt\right|,\]
where \textup{$\mathbf{U}$} can be taken as \textup{$\mathbf{U}_{1}$}
or \textup{$\mathbf{U}_{2}$}. Conversely, for any bounded linear
functional $L$ on $H^{1}(\mathbb{R})$, there is an $f\in BMO$ with
$\left\Vert L\right\Vert \eqsim\left\Vert f\right\Vert _{BMO}$.\\

\end{thm*}
\noindent We write $w=u+iv$ for the complex variable $w$ in what
follows. Let $\mathbb{C}^{\pm}=\{w:\pm v>0\}$ be the upper and lower
half planes, and let $P(w,t)=\frac{1}{\pi}\frac{v}{(u-t)^{2}+v^{2}}$
be the Poisson kernel on $\mathbb{C}^{+}$. Now define the square
$Q_{a,r}=\{w:a<u<a+r,0<v<r\}$. A measure $\mu$ on $\mathbb{C}^{+}$
is said to be a \textit{Carleson measure} if we have $\mathcal{N}(\mu):=\sup\left(\frac{\mu(Q_{a,r})}{r},a\in\mathbb{R},r>0\right)<\infty$.
In other words, the measure $\mu$ of any square protruding from the
real axis must be comparable to the length of its edge. The following
theorem characterizes $BMO$ in terms of such measures.
\begin{thm*}
\noindent \label{ThmBMO2}(Fefferman-Stein) Suppose $\int_{-\infty}^{\infty}\frac{|f(t)|}{t^{2}+1}dt<\infty$,
so that $P(w,\cdot)\star f$ is well-defined. Then\begin{equation}
\left\Vert f\right\Vert _{BMO}\eqsim\left[\mathcal{N}\left(v\left|\nabla_{u,v}(P(w,\cdot)\star f)\right|^{2}dudv\right)\right]^{1/2}.\label{FSThm}\end{equation}

\end{thm*}
\noindent A detailed discussion of $BMO$ and the significance of
these theorems can be found in \cite{Ga07} or \cite{St93}.

\section{Bandlimited Interpolation of Bounded Data\label{SecBounded}}

\noindent In this section, we establish a preliminary result showing
how adding an extra sample allows us to treat the bandlimited interpolation
of bounded data, such as the noise model discussed in Section \ref{SecIntro}.
We define

\noindent \begin{equation}
PW_{b}^{+}=\left\{ f\,\mathrm{entire}:\limsup_{r\to\infty}\int_{|z|=r}\left|z^{-2}e^{-\pi b|\mathrm{Im}(z)|}f(z)\right||dz|<\infty\right\} .\label{PWPlus}\end{equation}

\noindent The Plancherel-Polya inequality shows that $PW_{b}^{\infty}\subset PW_{b}^{+}$.
Functions in $PW_{b}^{+}$ can be expanded in the following way.
\begin{thm}
\label{ThmSampSeries2}Suppose $X=\{x_{k}\}\subset\mathbb{R}$ is
separated and has a sine-type generating function $S$ with $\mathrm{type}(S)=\pi b$,
and let \textup{$\tilde{x}\not\in X$.} If $f\in PW_{b}^{+}$and $\tilde{A}=f(\tilde{X})$,
then\begin{equation}
f(z)=\tilde{a}\frac{S(z)}{S(\tilde{x})}+\sum_{k=-\infty}^{\infty}a_{k}\lim_{z_{0}\to z}\frac{S(z_{0})}{S'(x_{k})}\left(\frac{1}{z_{0}-x_{k}}-\frac{1}{\tilde{x}-x_{k}}\right),\label{SampSeries2}\end{equation}

\noindent with uniform convergence of compact subsets of $\mathbb{C}$.
Conversely, for any $\tilde{A}\in l^{\infty}$, the series (\ref{SampSeries2})
converges uniformly on compact subsets of $\mathbb{C}$ and $f\in PW_{b}^{+}$.\end{thm}
\begin{proof}
\noindent We use a standard complex variable argument. Assume $z$
is in a closed ball $B$ with $z\not\in X$, and choose a real sequence
$\{r_{n}\}$ with $r_{n}\to\infty$ and $\mathrm{dist}(\{r_{n}\},X)>0$.
We can then consider the integral\[
J(r_{n}):=\frac{1}{2\pi i}\int_{|w|=r_{n}}\frac{f(w)S(z)}{S(w)}\left(\frac{1}{z-w}-\frac{1}{\tilde{x}-w}\right)|dw|.\]
For sufficiently large $n$, it can be seen by calculating residues
that\[
J(r_{n})=-f(z)+\tilde{a}\frac{S(z)}{S(\tilde{x})}+\sum_{|x_{k}|<r_{n}}a_{k}\frac{S(z)}{S'(x_{k})}\left(\frac{1}{z-x_{k}}-\frac{1}{\tilde{x}-x_{k}}\right).\]
The inequalities (\ref{SineType}) and (\ref{PWPlus}) imply that
as $r_{n}\to\infty$,\[
|J(r_{n})|\lesssim\max_{z\in B}|S(z)(z-\tilde{x})|\int_{|w|=r_{n}}\frac{|f(w)|e^{-\pi b|\mathrm{Im}(w)|}}{|w|^{2}}|dw|\to0.\]
By letting $z\to x_{k}$ for each $x_{k}\in B$, we obtain the formula
(\ref{SampSeries2}) for all $z\in B$. For the other direction of
Theorem \ref{ThmSampSeries2}, we note that $S$ has simple zeros
at exactly $X$, so for $z\in\mathbb{R}$, $|S(z)|\leq2||S'||_{L^{\infty}}\mathrm{dist}(z,X)$.
The Bernstein and Plancherel-Polya inequalities then show that for
$z\in\mathbb{\mathbb{C}}$ and $d=\sup_{k}|x_{k+1}-x_{k}|<\infty$,\[
|S(z)|\lesssim\left\Vert S\right\Vert _{L^{\infty}}\min(\mathrm{dist}(z,X),d)e^{\pi b|\mathrm{Im}(z)|}.\]
Now define the sets:\begin{eqnarray*}
I_{1}^{w} & = & (\lfloor\mathrm{Re}(w)\rfloor-\min(1/2,\lambda(X)),\lfloor\mathrm{Re}(w)\rfloor+\min(1/2,\lambda(X)))\\
I_{2} & = & \left(-\infty,\lfloor(\mathrm{Re}(z)+\tilde{x})/2\rfloor\right)\backslash(I_{1}^{z}\bigcup I_{1}^{\tilde{x}})\\
I_{3} & = & \left(\lfloor(\mathrm{Re}(z)+\tilde{x})/2\rfloor+1,\infty\right)\backslash(I_{1}^{z}\bigcup I_{1}^{\tilde{x}})\end{eqnarray*}
Using the separation of $X$ along with basic properties of lower
Riemann sums, we have\begin{eqnarray}
|f(z)| & \lesssim & |\tilde{a}S(z)|+e^{\pi b|\mathrm{Im}(z)|}\left\Vert A\right\Vert _{l^{\infty}}\sum_{k=-\infty}^{\infty}\frac{\min(\mathrm{dist}(z,X),d)|z-\tilde{x}|}{|z-x_{k}||\tilde{x}-x_{k}|}\nonumber \\
 & \lesssim & \left\Vert \tilde{A}\right\Vert _{l^{\infty}}e^{\pi b|\mathrm{Im}(z)|}\left(1+\sum_{k\in\mathbb{Z}\bigcap I_{2}}\frac{|x_{k+1}-x_{k}||z-\tilde{x}|}{\lambda(X)|z-x_{k}||\tilde{x}-x_{k}|}+\sum_{k\in\mathbb{Z}\bigcap I_{3}}\frac{|x_{k}-x_{k-1}||z-\tilde{x}|}{\lambda(X)|z-x_{k}||\tilde{x}-x_{k}|}\right)\nonumber \\
 & \lesssim & \left\Vert \tilde{A}\right\Vert _{l^{\infty}}e^{\pi b|\mathrm{Im}(z)|}\left(1+\int_{\mathbb{R}\backslash(I_{1}^{z}\bigcup I_{1}^{\tilde{x}})}\frac{|z-\tilde{x}|}{|z-t||\tilde{x}-t|}dt\right)\nonumber \\
 & \lesssim & \left\Vert \tilde{A}\right\Vert _{l^{\infty}}e^{\pi b|\mathrm{Im}(z)|}(1+\max(\log|z|,0)),\label{LogBound}\end{eqnarray}
which implies that $f\in PW_{b}^{+}$.
\end{proof}
\noindent This expansion can be compared with the series (\ref{SampSeries1}).
It is essentially a nonuniform version of the classical Valiron interpolation
formula considered in \cite{BM07}, in which the derivative of $f$
at a point is used instead of the extra sample $\tilde{a}$, but the
form considered here will be more convenient for our purposes. We
also mention that the extra point $\tilde{x}$ plays no special role
in the collection $\tilde{X}$, and we isolate it mainly for notational
convenience. If we pick any point $x_{j}\in X$ and let $y_{k}=x_{k}$
for $k\not=j$, $y_{j}=\tilde{x}$ and $\tilde{y}=x_{j}$, then $\tilde{Y}=\{y_{k}\}\bigcup\tilde{y}$
satisfies the conditions of Theorem \ref{ThmSampSeries2} too.\\

\noindent For any $\tilde{A}\in l^{\infty}$, we call the function
$f$ given by (\ref{SampSeries2}) the \textit{bandlimited interpolation}
of $\tilde{A}$ at $\tilde{X}$. Note that for any given $\tilde{a}_{2}$
and $\tilde{x}_{2}\not\in X$, if $g$ is the bandlimited interpolation
of $A\bigcup\{\tilde{a}_{2}\}$ at $X\bigcup\{\tilde{x}_{2}\}$, then
$g(z)=f(z)+cS(z)$ for some constant $c$. Moreover, if $A\in l^{2}$,
then for any given $\tilde{x}\not\in X$ we can always choose $\tilde{a}$
so that $f$ coincides with the series (\ref{SampSeries1}), or in
the special case of uniformly spaced points $X=\{\frac{k}{b}\}$,
the usual bandlimited interpolation given by the WSK series (\ref{WSKSeries}).\\

We discuss an example of a $PW_{1}^{+}$ function that illustrates
many of the typical properties of the series (\ref{SampSeries2}).
We use the uniform samples $X=\{k\}$ and denote $\psi(z)=\frac{\Gamma'(z)}{\Gamma(z)}$,
where $\Gamma$ is the usual gamma function. The properties of $\psi$
are discussed in depth in \cite{WW27}.
\begin{example*}
\noindent \label{ExG1}The function $G_{1}(z)=\sin(\pi z)\psi(-z)$
is in $PW_{1}^{+}\backslash PW_{1}^{\infty}$ and satisfies $a_{k}=0$
for $k<0$ and $a_{k}=(-1)^{k}\pi$ for $k\geq0$.
\end{example*}
\noindent The function $\psi$ satisfies the estimate\begin{equation}
\lim_{|z|\to\infty,|\arg z|<\pi}\frac{\psi(z)}{\log z}=1,\label{PsiFormula1}\end{equation}

\noindent so $G_{1}$ is not bounded. With $A=\{a_{k}\}$ given as
above, Theorem \ref{ThmSampSeries2} shows that for any $\tilde{x}$
and $\tilde{a}$, the (unique) bandlimited interpolation of $\tilde{A}$
at $\tilde{X}$ is of the form $G_{1}(z)+c\sin(\pi z)$. It follows
that the samples $A$ have no bandlimited interpolation in $PW_{1}^{\infty}$.\\

\noindent It will be instructive to isolate one property of $G_{1}$
here. A classical formula of Gauss (\cite{WW27}, p. 240) shows that
for integer $k>0$,\begin{equation}
G_{1}(k-\frac{1}{2})=G_{1}(-k-\frac{1}{2})=(-1)^{k}\left(\sum_{m=1}^{k-1}\frac{1}{m}+\sum_{m=k}^{2k-1}\frac{2}{m}+C\right),\label{PsiFormula2}\end{equation}
so as $z\to\infty$, $|G_{1}(z)|$ grows logarithmically in between
the integer samples. The same applies as $z\to-\infty$, even though
the samples at $k<0$ are all zero. This can be interpreted as a \textit{nonlocal
effect}, where the sustained growth of $|G_{1}|$ on the positive
real axis, caused by the {}``bad behavior'' of the samples at $k>0$,
induces growth on the negative real axis too. This property can be
seen in the graph of $G_{1}$ in Figure \ref{FigG1}. It is also present
in the bandlimited interpolation of Boche and Mönich's example $a_{k}=0$,
$k<1$, and $a_{k}=(-1)^{k}/\log(k+1)$, $k\geq1$, where we take
$\tilde{x}=\frac{1}{2}$ and $\tilde{a}=0$.

\begin{figure}[H]
\subfloat{\includegraphics[scale=0.75]{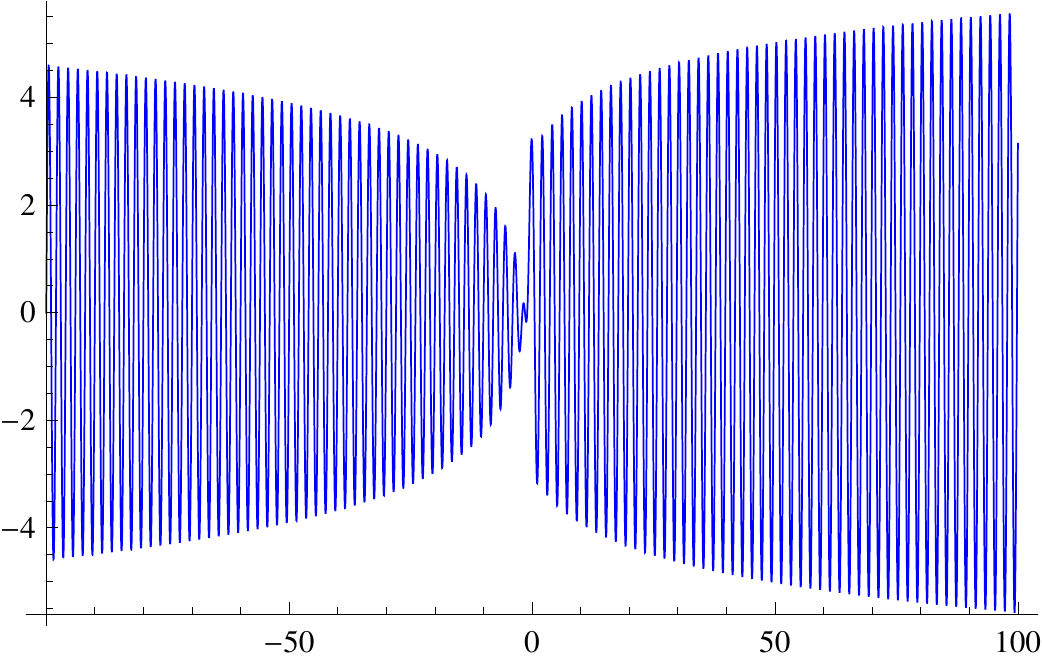}}~~~~~\subfloat{\includegraphics[scale=0.75]{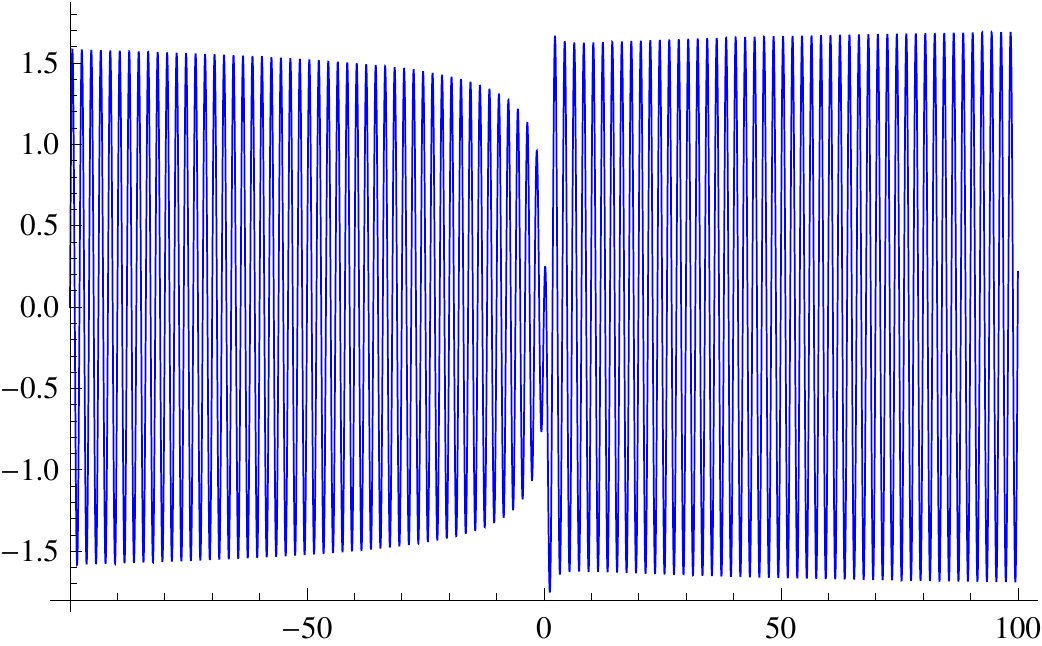}}

\caption{\label{FigG1}Left: The function $G_{1}(z)$. Right: The bandlimited
interpolation of Boche and Mönich's sequence.}

\end{figure}

\section{Bandlimited Interpolation of Random Data\label{SecRandom}}

We can now state the main result of this paper.
\begin{thm}
\label{ThmRandom}Suppose $X\subset\mathbb{R}$ is separated and has
a sine-type generating function $S$ with $\mathrm{type}(S)\leq\pi b$,
and let $\tilde{x}\not\in X$. Suppose also that $\tilde{A}=\{a_{k}\}\bigcup\tilde{a}$
is a collection of i.i.d. random variables uniformly distributed in
$[-\alpha,\alpha]$. Let $f$ be the bandlimited interpolation of
$\tilde{A}$ at $\tilde{X}$. Then almost surely,\begin{equation}
\sup_{r>0}\frac{1}{2r}\int_{-r}^{r}|f(t)|dt<\infty.\label{Average}\end{equation}

\end{thm}
We make a few comments before proving Theorem \ref{ThmRandom}. This
result deals with the same situation discussed in Section \ref{SecIntro},
even though it has been formulated slightly differently. In the notation
of Section \ref{SecIntro}, we can take $T_{k}$ to be zero by linearity
and only consider the noise $N_{k}$. As we saw in Section \ref{SecBounded},
the extra sample $\tilde{a}$ can be taken as deterministic and changed
arbitrarily without affecting the result of Theorem \ref{ThmRandom}.
The exact probability distribution of $\tilde{A}$ is also of little
significance here, and the result holds more generally for any symmetric,
finitely supported distribution.\\

\noindent We split the proof of Theorem \ref{ThmRandom} into three
lemmas for clarity. Our approach is to write the function $f$ as
the sum of two parts, each with only zero samples in one direction
along the real axis, and show that each one is almost surely bounded
on that side. This shows directly that the nonlocal effect discussed
in Section \ref{SecBounded} does not occur. We then move to the deterministic
setting and show that this one-sided boundedness forces a certain
regularity upon the other side, resulting in the function having a
bounded global average.\\

\noindent For the rest of this section, we assume that $\tilde{X}$
and $S$ are as given in Theorem \ref{ThmRandom}, without repeating
the conditions on them every time.
\begin{lem}
\label{ThmRandP1}For $k$ such that $x_{k}>0$, let $\{a_{k}\}$
be a collection of i.i.d. random variables uniformly distributed in
$[-\alpha,\alpha]$, let $a_{k}=0$ for all other $k$ and let $\tilde{a}=0$.
Suppose $f$ is the bandlimited interpolation of $\tilde{A}$ at $\tilde{X}$.
Then \textup{$\sup_{t<0}|f(t)|<\infty$} almost surely.\end{lem}
\begin{proof}
We can assume that $x_{0}=\min(x_{k}:x_{k}>0)$ and $\tilde{x}>0$,
as the general case follows from the remarks after Theorem \ref{ThmSampSeries2}.
Let $b_{k}=\frac{a_{k}}{S'(x_{k})(\tilde{x}-x_{k})}$. Then we have\[
\sum_{k=0}^{\infty}E\left(b_{k}\right)=0\]
and the separation property shows that for some constant $d$,\begin{eqnarray*}
\sum_{k=0}^{\infty}\mathrm{var}\left(b_{k}\right) & = & \frac{\alpha^{2}}{3}\sum_{k=0}^{\infty}\frac{1}{S'(x_{k})^{2}(\tilde{x}-x_{k})^{2}}\\
 & \lesssim & \frac{\alpha^{2}}{3}\sum_{k=0}^{\infty}\frac{1}{(\mathrm{dist}(\tilde{x},X)+\lambda(X)|k-d|)^{2}}\\
 & < & \infty.\end{eqnarray*}
By Kolmogorov's three-series theorem, $\sum_{k=0}^{\infty}b_{k}$
converges almost surely. Now let\[
g(t)=\frac{f(t)}{S(t)}=\sum_{k=0}^{\infty}\frac{a_{k}}{S'(x_{k})}\left(\frac{1}{t-x_{k}}-\frac{1}{\tilde{x}-x_{k}}\right).\]
It is easy to check that if $\sum_{k=0}^{\infty}b_{k}$ converges,
then $\lim_{t\to-\infty}g(t)=\sum_{k=0}^{\infty}b_{k}$. Since $|g(0)|<\infty$,
it follows by continuity that $\sup_{t<0}|g(t)|<\infty$ almost surely.
We also have $\sup_{t<0}|f(t)|\lesssim\sup_{t<0}|g(t)|$, which proves
the lemma.\end{proof}
\begin{lem}
\label{ThmRandP2}For any $\tilde{A}\in l^{\infty}$, let $f$ be
the bandlimited interpolation of $\tilde{A}$ at $\tilde{X}$. Then
for each $c>0$,\[
\left\Vert \frac{f(\cdot+ic)}{S(\cdot+ic)}\right\Vert _{BMO}\lesssim\left\Vert A\right\Vert _{l^{\infty}}.\]
\end{lem}
\begin{proof}
Applying Fefferman's duality theorem to the series (\ref{SampSeries2})
gives\begin{eqnarray*}
\left\Vert \frac{f(\cdot+ic)}{S(\cdot+ic)}\right\Vert _{BMO} & \lesssim & \sup_{h\in\mathbf{U}_{1}}\frac{1}{\left\Vert h\right\Vert _{H^{1}(\mathbb{R})}}\left|\int_{-\infty}^{\infty}\sum_{k=-\infty}^{\infty}\frac{a_{k}h(z)}{S'(x_{k})}\left(\frac{1}{z+ic-x_{k}}-\frac{1}{\tilde{x}-x_{k}}\right)dz\right|.\end{eqnarray*}

\noindent Since $h$ is finitely supported and the series (\ref{SampSeries2})
converges uniformly on compact sets, we can interchange the order
of summation and integration. $\mathcal{P}^{+}h$ and $\mathcal{P}^{-}h$
are in $L^{1}$, so by analyticity we have\begin{eqnarray*}
\left\Vert \frac{f(\cdot+ic)}{S(\cdot+ic)}\right\Vert _{BMO} & \lesssim & \sup_{h\in\mathbf{U}_{1}}\frac{1}{\left\Vert h\right\Vert _{H^{1}(\mathbb{R})}}\left|\sum_{k=-\infty}^{\infty}\frac{a_{k}}{S'(x_{k})}\int_{-\infty}^{\infty}\left(\frac{\mathcal{P}^{+}h(z)+\mathcal{P}^{-}h(z)}{z+ic-x_{k}}-\frac{h(z)}{\tilde{x}-x_{k}}\right)dz\right|\\
 & = & \sup_{h\in\mathbf{U}_{1}}\left|\frac{2\pi i}{\left\Vert h\right\Vert _{H^{1}(\mathbb{R})}}\sum_{k=-\infty}^{\infty}\frac{a_{k}\mathcal{P}^{-}h(x_{k}-ic)}{S'(x_{k})}\right|\\
 & \lesssim & \left\Vert A\right\Vert _{l^{\infty}}\sup_{h\in\mathbf{U}_{1}}\frac{1}{\left\Vert h\right\Vert _{H^{1}(\mathbb{R})}}\sum_{k=-\infty}^{\infty}\left|\mathcal{P}^{-}h(x_{k}-ic)\right|.\end{eqnarray*}
Since $X$ is separated, an elementary property of Hardy spaces (\cite{Le96},
p. 138) is that\[
\sum_{k=-\infty}^{\infty}\left|\mathcal{P}^{-}h(x_{k}-ic)\right|\lesssim\left\Vert \mathcal{P}^{-}h\right\Vert _{L^{1}}\leq\left\Vert h\right\Vert _{H^{1}(\mathbb{R})},\]
which completes the proof.\end{proof}
\begin{lem}
\label{ThmRandP3}For any $\tilde{A}\in l^{\infty}$, let $f$ be
the bandlimited interpolation of $\tilde{A}$ at $\tilde{X}$. Suppose
that $\sup_{t<0}|f(t)|<\infty$ and for some $c>0$, $\frac{f(\cdot+ic)}{S(\cdot+ic)}\in BMO$.
Then $\sup_{r>0}\frac{1}{2r}\int_{-r}^{r}|f(t)|dt<\infty$.\end{lem}
\begin{proof}
We assume $c=1$ without loss of generality. Let $f^{\pm}(z)=f(z)e^{\pm\pi biz}$,
$g(z)=\frac{f(z+i)}{S(z+i)}$, $M_{1}=\sup_{t<0}|f(t)|$ and $M_{2}=\sup_{t<0}|g(t)|$.
The estimate (\ref{LogBound}) implies that $\int_{-\infty}^{\infty}\frac{|f(t)|}{t^{2}+1}<\infty$,
so $|f^{+}|$ has a harmonic majorant on the upper half plane (see
\cite{Ga07}) and the reproducing formula $f^{+}(z)=P(z,\cdot)\star f^{+}$
holds for $\mathrm{Im}(z)>0$. We can then estimate\begin{eqnarray*}
\sup_{t<0}|f^{+}(t+i)| & \leq & \sup_{t<0}\left(M_{1}\int_{-\infty}^{0}P(t+i,s)ds+\int_{0}^{\infty}|f(s)|P(t+i,s)ds\right)\\
 & \leq & \left(\frac{M_{1}}{2}+\frac{1}{\pi}\int_{0}^{\infty}\frac{|f(s)|}{s^{2}+1}ds\right).\end{eqnarray*}

\noindent This shows that $M_{2}<\infty$. Now for any fixed $r>0$,\begin{eqnarray*}
\frac{1}{2r}\int_{-r}^{r}|f(t+i)|dt & \lesssim & \frac{1}{2r}\int_{-r}^{r}|g(t)|dt\\
 & \leq & \frac{1}{2r}\left(\int_{0}^{r}|g(t)|dt-\int_{-r}^{0}|g(t)|dt\right)+M_{2}\\
 & \leq & \frac{1}{2r}\left(\left|\int_{0}^{r}|g(t)|dt-\int_{-r}^{r}|g(s)|ds\right|+\left|\int_{-r}^{0}|g(t)|dt-\int_{-r}^{r}|g(s)|ds\right|\right)+M_{2}\\
 & \leq & \frac{1}{r}\int_{-r}^{r}\left|g(t)-\frac{1}{2r}\int_{-r}^{r}g(s)ds\right|dt+M_{2}\\
 & \leq & 2\left\Vert g\right\Vert _{BMO}+M_{2}.\end{eqnarray*}

\noindent We finally use a Poisson integral again to move back to
the real line. For $\mathrm{Im}(z)<1$, we have $f^{-}(z)=P(z-i,\cdot)\star f^{-}(\cdot+i)$.
This gives\begin{eqnarray*}
\frac{1}{2r}\int_{-r}^{r}|f(t)|dt & \leq & e^{\pi b}\frac{1}{2r}\int_{-r}^{r}\int_{-\infty}^{\infty}\frac{|f(s+i)|}{(t-s)^{2}+1}dsdt\\
 & = & e^{\pi b}\frac{1}{2\pi r}\int_{-\infty}^{\infty}|f(s+i)|\left(\arctan\left(r+s\right)+\arctan(r-s)\right)ds\\
 & \leq & e^{\pi b}\left(\frac{1}{2r}\int_{-2r}^{2r}|f(s+i)|ds+2\int_{\mathbb{R}\backslash[-2r,2r]}\frac{|f(s+i)|}{s^{2}+1}ds\right).\end{eqnarray*}

\noindent Taking the estimate (\ref{LogBound}) into account again,
we conclude that\[
\sup_{r>0}\frac{1}{2r}\int_{-r}^{r}|f(t)|dt<\infty.\]

\end{proof}
\noindent We can now combine these lemmas to complete the proof.
\begin{proof}[Proof of Theorem \ref{ThmRandom}]
\noindent  For any $\tilde{A}\in l^{\infty}$, we can write the bandlimited
interpolation $f$ of $\tilde{A}$ at $\tilde{X}$ as $f(z)=f_{1}(z)+f_{2}(z)+\frac{\tilde{a}S(z)}{S(\tilde{x})}$,
where $f_{1}(x_{k})=0$ for $x_{k}<0$ and $f_{2}(x_{k})=0$ for $x_{k}\geq0$.
Applying Lemmas \ref{ThmRandP1}-\ref{ThmRandP3} on $f_{1}(z)$ and
$f_{2}(-z)$ and noting that $S\in L^{\infty}$ finishes the proof.
\end{proof}
The statistical incoherence in the samples $\tilde{A}$ in Theorem
\ref{ThmRandom} is the reason we have the bounded average property
(\ref{Average}), and it does not generally hold for bounded samples
$\tilde{A}$. As an illustration of this, we return to the example
function $G_{1}$ from Section \ref{SecBounded} and show that the
average of $|G_{1}(t)|$ is unbounded. It suffices to consider $t<0$.
Let $T$ be the tent function\begin{equation}
T(t)=\left\{ \begin{array}{cc}
2t & 0<t\leq\frac{1}{2}\\
2-2t & \frac{1}{2}<t\leq1\\
0 & \mathrm{otherwise}\end{array}\right..\label{Tent}\end{equation}

\noindent It is clear that $|\sin(\pi t)|\geq\sum_{n=-\infty}^{\infty}T(t+n)$,
and the formula (\ref{PsiFormula1}) implies that $|\psi(t)|\geq\frac{1}{2}\log|t|$
for sufficiently large $t$. This shows that\[
|G_{1}(t)|\geq\sum_{n=2}^{\infty}\frac{1}{2}\log(n)T(t+n).\]

\noindent It follows that as $r\to\infty$, $\frac{1}{r}\int_{-r}^{0}|G_{1}(t)|dt\gtrsim\log r\to\infty$.\\

\noindent Figure \ref{FigG2} below shows an example of the bandlimited
interpolation of random data. In the notation of Theorem \ref{ThmRandom},
we use a realization of $\tilde{A}$ with $\alpha=\frac{1}{2}$, and
take $x_{k}=k$ and $\tilde{x}=\frac{1}{2}$. We denote the resulting
function by $G_{2}$. The graphs in Figure \ref{FigG2} can be compared
with the functions shown in Figure \ref{FigG1} in Section \ref{SecBounded}.
Unlike those functions, it can be seen that $G_{2}$ does not steadily
grow over long time intervals. Intuitively, this shows how the effect
of noisy samples on the bandlimited interpolation is in a sense well-controlled.\\

\begin{figure}[H]
\subfloat{\includegraphics[scale=0.75]{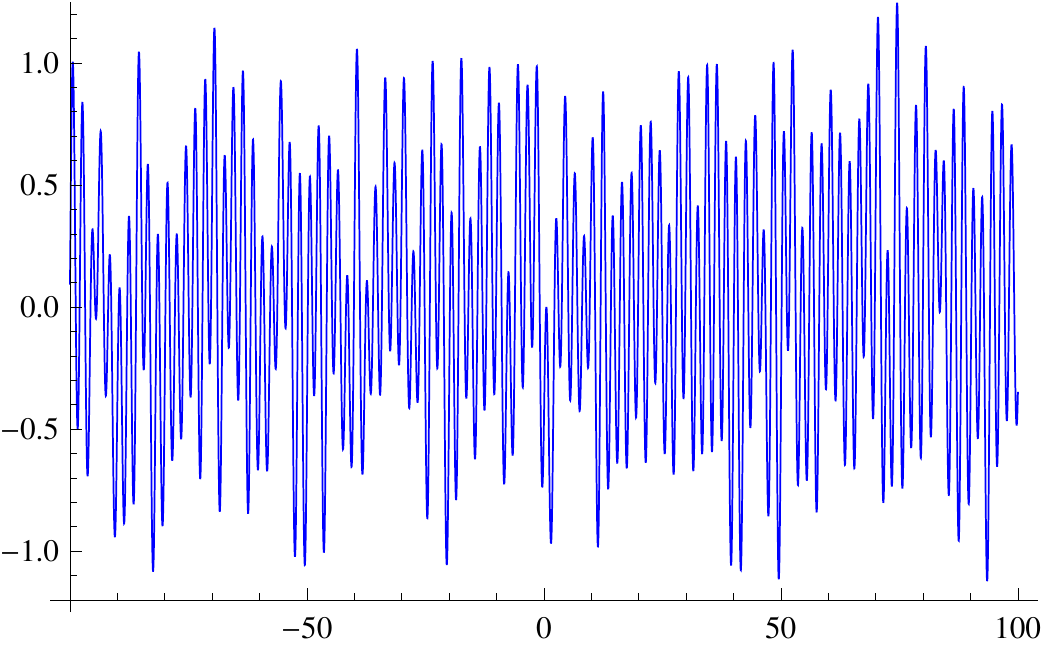}}~~~~~\subfloat{\includegraphics[scale=0.75]{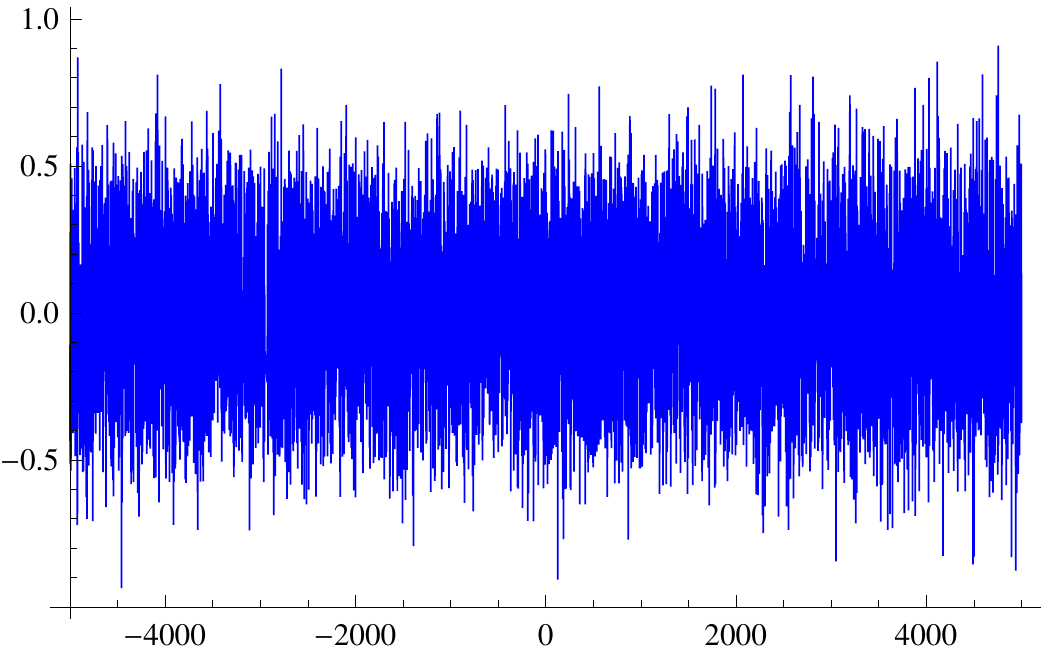}}

\caption{\label{FigG2}Left: The function $G_{2}(z)$ on $[-100,100]$. Right:
$G_{2}(z)$ on $[-5000,5000]$.}

\end{figure}

\section{Bandlimited BMO Functions\label{SecBMO}}

In this section, we study some properties of bandlimited functions
in the space $BMO$. Such functions have a somewhat different character
than the examples we have seen so far. We fix a point $c$ and define
the space $PW_{b}^{\star}$ to be the following.\[
PW_{b}^{\star}=\{f:\mathrm{type}(f)\leq\pi b,\left\Vert f\right\Vert _{BMO,c}:=|f(c)|+\left\Vert f\right\Vert _{BMO}<\infty\}\]

\noindent The term $|f(c)|$ resolves the ambiguity in the $BMO$
seminorm for constant functions, and $\left\Vert f\right\Vert _{BMO,c}$
is a (full) norm. It will be shown below that the precise value of
$c$ is unimportant and that changing it gives an equivalent norm.
Since $f\in BMO$ always satisfies $\int_{-\infty}^{\infty}\frac{|f(t)|}{t^{2}+1}dt<\infty$
\cite{Ga07}, the Paley-Wiener-Schwartz theorem implies that $PW_{b}^{\star}\subset PW_{b}^{+}$.
We first give a version of the Plancherel-Polya inequality for $PW_{b}^{\star}$.
\begin{lem}
\label{ThmPPBMO}If $f\in PW_{b}^{\text{\ensuremath{\star}}}$, then
$\left\Vert f(\cdot+ic)\right\Vert _{BMO}\leq\left\Vert f\right\Vert _{BMO}e^{\pi b|c|}$.\end{lem}
\begin{proof}
The proof is similar to the $PW_{b}^{p}$ case described in \cite{Se04}.
Define\[
R_{\epsilon}^{\pm}(z)=e^{\mp(\pi b+\epsilon)\mathrm{Im}(z)}\frac{1}{2r}\intop_{-r}^{r}\left|f(z+t)-\frac{1}{2r}\intop_{-r}^{r}f(z+s)ds\right|dt,\]
for complex $z$ and real $r.$ For each $\epsilon>0$, $R_{\epsilon}^{+}$
is a subharmonic function satisfying $\left|R_{\epsilon}^{+}(z)\right|\leq\left\Vert f\right\Vert _{BMO}$
for $z\in\mathbb{R}$ and $\max(\log|R_{\epsilon}^{+}(z)|,0)\to0$
as $z\to i\infty$. Applying the Phragmen-Lindelöf principle over
$\mathbb{C}^{+}$gives $|R_{\epsilon}^{+}(z+ic)|\leq\left\Vert f\right\Vert _{BMO}e^{(\pi b+\epsilon)|c|}$
for $c\geq0$, and we can repeat the argument with $R_{\epsilon}^{-}$
and $\mathbb{C}^{-}$ for $c<0$. Taking the supremum over real $z$
and $r$ and letting $\epsilon\rightarrow0$ gives the inequality.
\end{proof}
We will now establish several basic properties of $PW_{b}^{\star}$.
\begin{thm}
\label{ThmPWStar}Let $f\in PW_{b}^{\text{\ensuremath{\star}}}$.
Then the following statements hold.

I: For each $c\in\mathbb{R}$, $f(\cdot+ic)$ is uniformly Lipschitz
continuous on $\mathbb{\mathbb{R}}$.

II: For any fixed numbers $c$ and $c'$,$\left\Vert f\right\Vert _{BMO,c}\eqsim\left\Vert f\right\Vert _{BMO,c'}$.

III: For any given $z\in\mathbb{C}$, the point evaluation functional
$z\rightarrow f(z)$ is bounded on \textup{$PW_{b}^{\text{\ensuremath{\star}}}$}.

IV: $\left\Vert f'\right\Vert _{L^{\infty}}\lesssim\left\Vert f\right\Vert _{BMO}$.\end{thm}
\begin{proof}
We set $b=1$ without loss of generality. We can prove all of the
above statements by using the reproducing kernel-like function\[
K(c,t)=\frac{|c|}{\pi t(t-c)}\sin\left(\frac{2\pi N}{c}t\right),\]
where $c\in\mathbb{R}\backslash\{0\}$ and $N$ is any integer greater
than $|c|$. As a function of $t$, $K(c,t)$ is entire and satisfies
$2\pi\leq\mathrm{type}(K)<\infty$. For any $f\in PW_{1}^{+}$,\[
\int_{-\infty}^{\infty}f(t)K(c,t)dt=f(c)-f(0).\]
This can be seen by observing that for $\eta=\pm1$, the function
$\frac{c}{z(z-c)}\exp\left(2\pi i\eta Nz/c\right)f(z)$ has poles
at $c$ and $0$ with respective residues $\eta f(c)$ and $-\eta f(0)$.
The estimation argument is very similar to the proof of Theorem \ref{ThmSampSeries2},
and we omit the details.\\

We now suppose that $f\in PW_{1}^{\star}$. We want to approximate
the $H^{1}(\mathbb{R})$ norm of $K(c,t)-K(c',t)$, where $c\geq1$
and $c'\geq1$. We first integrate the function $\frac{1}{\pi(z-s)}\frac{c}{z(z-c)}\exp\left(2\pi i\eta Nz/c\right)$,
where $s\in\mathbb{R}\backslash\{0,c\}$, and perform the same kind
of calculation as before to find that\begin{eqnarray*}
\mathcal{H}K(c,s) & = & -\frac{1}{\pi s}-\frac{1}{\pi(c-s)}+\frac{c\exp(2\pi iNs/c)}{2\pi s(c-s)}+\frac{c\exp(-2\pi iNs/c)}{2\pi s(c-s)}\\
 & = & \frac{c\left(\cos(2\pi Ns/c)-1\right)}{\pi s(c-s)}.\end{eqnarray*}
Let $N=\max(\lceil c\rceil,\lceil c'\rceil)$ and define the interval
$I^{w}:=[w-\frac{1}{2},w+\frac{1}{2}]$. We first consider the case
where $1\leq c\leq\frac{3}{2}$ and $|c-c'|>\frac{1}{2}$. Recalling
that $T$ is the tent function (\ref{Tent}), we have\begin{eqnarray*}
 &  & \left\Vert K(c,\cdot)-K(c',\cdot)\right\Vert _{L^{1}}\\
 & \leq & \int_{-\infty}^{\infty}\sum_{n=-\infty}^{\infty}\left(\frac{2cT(2Nt/c+n)}{\pi\left|t(t-c)\right|}+\frac{2c'T(2Nt/c'+n)}{\pi\left|t(t-c')\right|}\right)dt\\
 & \leq & \frac{16N}{\pi}+\int_{\mathbb{R}\backslash(I^{0}\bigcup I^{c})}\frac{2c}{\pi\left|t(t-c)\right|}dt+\int_{\mathbb{R}\backslash(I^{0}\bigcup I^{c'})}\frac{2c'}{\pi\left|t(t-c')\right|}dt\\
 & \lesssim & |c-c'|.\end{eqnarray*}

\noindent Now suppose that $1\leq c\leq\frac{3}{2}$ and $|c-c'|\leq\frac{1}{2}$,
so that $N=2$. Some elementary estimates show that\begin{eqnarray*}
 &  & \left\Vert K(c,\cdot)-K(c',\cdot)\right\Vert _{L^{1}}\\
 & \lesssim & \int_{-1/2}^{1/2}\left|\frac{4}{c}-\frac{4}{c'}\right|dt+\int_{1/2}^{5/2}\max\left(\left|\frac{4}{c}-\frac{\sin(4\pi c/c')}{\pi(c-c')}\right|,\left|\frac{4}{c'}-\frac{\sin(4\pi c'/c)}{\pi(c'-c)}\right|\right)dt+\\
 &  & \quad\int_{\mathbb{R}\backslash(-1/2,5/2)}|c-c'||t|^{-3/2}dt\\
 & \lesssim & |c-c'|.\end{eqnarray*}

\noindent Following the same arguments, we can also obtain the bound
$\Vert\mathcal{H}K(c,\cdot)-\mathcal{H}K(c',\cdot)\Vert_{L^{1}}\lesssim|c-c'|$
for the above choices of $c$ and $c'$. By Fefferman's duality theorem
and the fact that $K(c,\cdot)-K(c',\cdot)\in\mathbf{U}_{2}$, we have\begin{equation}
\left\Vert f\right\Vert _{BMO}\gtrsim\frac{1}{\left\Vert K(c,\cdot)-K(c',\cdot)\right\Vert _{H^{1}(\mathbb{R})}}\left|\int_{-\infty}^{\infty}f(t)(K(c,t)-K(c',t))dt\right|\gtrsim\left|\frac{f(c)-f(c')}{c-c'}\right|,\label{PWStarLip}\end{equation}
where the constant in the inequality is independent of $c$ and $c'$.
Since the $BMO$ seminorm is translation-invariant, the inequality
(\ref{PWStarLip}) actually holds for all $c,c'\in\mathbb{R}$. Combining
this with Lemma \ref{ThmPPBMO} proves (I) and letting $c'\to c$
gives (IV). If we fix $R=|c-c'|>0$, this also shows that $\left\Vert f\right\Vert _{BMO}+|f(c)|\gtrsim|f(c')|$,
where the implied constant depends on $R$, and we can interchange
$c$ and $c'$ to get (II). Finally, the statement (III) is just (II)
phrased in a different way.\end{proof}
\begin{rem*}
The closure of the set of uniformly continuous $BMO$ functions under
the $BMO$ seminorm is called $VMO$, for vanishing mean oscillation.
Theorem \ref{ThmPWStar} (I) shows that $PW_{b}^{\star}\subset VMO$.
Note that there are two non-equivalent definitions of $VMO$ in the
literature, and we use the one given in \cite{Ga07}.
\end{rem*}
~
\begin{rem*}
Theorem \ref{ThmPWStar} (IV) is a sharper form of the $p=\infty$
case of Bernstein's inequality. We mention that the opposite inequality
does not generally hold (even if $\left\Vert f\right\Vert _{BMO}$
is replaced by $\left\Vert f\right\Vert _{BMO,c}$), and there are
functions $f$ such that $f'\in PW_{b}^{\infty}$ but $f\not\in PW_{b}^{\text{\ensuremath{\star}}}$.\end{rem*}
\begin{cor}
\label{ThmUnbound}Let $f\in PW_{b}^{\star}$. Then either $f\in PW_{b}^{\infty}$
or there is no separated sequence $X$ with $D^{-}(X)>0$ such that
$f(X)\in l^{\infty}$.\end{cor}
\begin{proof}
Suppose we have a separated $X=\{x_{k}\}$ with $D^{-}(X)>0$ and
$f(X)\in l^{\infty}$. This means that for some large fixed $r$,
every real interval $I$ of length $r$ contains a point $x_{n}\in X$.
Theorem \ref{ThmPWStar} (IV) then shows that for any $t\in I$,

\[
\left|f(t)\right|=\left|f(x_{n})+\int_{x_{n}}^{t}f'(u)du\right|\lesssim\left|f(x_{n})\right|+r\left\Vert f\right\Vert _{BMO}.\]

\end{proof}
Intuitively, Corollary \ref{ThmUnbound} says that an unbounded $PW_{b}^{\star}$
function is large in most places on the real line. It also shows that
the bandlimited interpolation of bounded data $\tilde{A}\in l^{\infty}$
can never be in $PW_{b}^{\star}$ unless it is actually in $PW_{b}^{\infty}$.
This occurs in spite of Lemma \ref{ThmRandP2} and highlights a basic
difference between $PW_{b}^{\star}$ and $PW_{b}^{p}$, $1<p<\infty$.
In Lemma \ref{ThmRandP2}, we generally cannot remove the factor $\frac{1}{S(\cdot+ic)}$
from the inequality and conclude that $f\in BMO$. In contrast, for
$A\in l^{p}$, the series (\ref{SampSeries1}) can be used to find
that $\frac{f(\cdot+ic)}{S(\cdot+ic)}\in L^{p}$ (see \cite{Le96}),
which clearly implies $f(\cdot+ic)\in L^{p}$ and thus $f\in L^{p}$.\\

We finally study an example of an unbounded $PW_{b}^{\star}$ function
that illustrates the {}``largeness'' property described above.
\begin{example*}
\label{ExG3}The function $G_{3}(z)=\sum_{k=0}^{\infty}(-1)^{k}\sin\left(\frac{\pi z}{3\cdot2^{k}}\right)$
is in $PW_{1/3}^{\star}\backslash PW_{1/3}^{\infty}$.
\end{example*}
To see this, we use the identity $\sin z=\frac{1}{2i}\left(e^{iz}-e^{-iz}\right)$
to write $G_{3}=G_{3+}+G_{3-}$, where $P(w,\cdot)\star G_{3\pm}=G_{3\pm}(w)$
for $w\in\mathbb{C}^{\pm}$, and then apply the Fefferman-Stein theorem
(\ref{FSThm}) to each part. Let $w=u+iv$. We first note that by
analyticity,\[
\left|\nabla(P(w,\cdot)\star G_{3+})\right|^{2}=\left|\nabla G_{3+}(u+iv)\right|^{2}=2|G_{3+}'(w)|^{2}.\]

\noindent Since $\left|\sin\frac{\pi z}{3\cdot2^{k}}\right|\leq\left|\frac{\pi z}{3\cdot2^{k}}\right|$
for large $k$, the series defining $G_{3}$ converges uniformly on
compact sets, so we have\begin{eqnarray*}
\mathcal{N}\left(2v\left|G_{3+}'(w)\right|^{2}dudv\right) & = & \mathcal{N}\left(2v\left|\frac{d}{dw}\frac{1}{2i}\sum_{k=0}^{\infty}e^{\frac{\pi i}{3}2^{-k}w}\right|^{2}dudv\right)\\
 & \leq & \mathcal{N}\left(2ve^{-\frac{2\pi}{3}v}\left(\sum_{k=0}^{\infty}\frac{\pi}{3\cdot2^{k}}\right)^{2}dudv\right)\\
 & \leq & 2.\end{eqnarray*}

\noindent Doing the same calculation with $G_{3-}$, we find that
$G_{3}\in PW_{1/3}^{\star}$. On the other hand, $G_{3}$ satisfies
the identity $G_{3}(2z)=\sin(\frac{2\pi z}{3})-G_{3}(z)$. This implies
that for integer $n\geq2$,\begin{eqnarray*}
G_{3}(2^{n}) & = & (-1)^{n}g_{3}(1)+\sum_{k=0}^{n-1}(-1)^{n-k}\sin\left(\frac{\pi2^{k}}{3}\right)\\
 & = & (-1)^{n}\left(g_{3}(1)-\frac{\sqrt{3}}{2}(n-2)\right),\end{eqnarray*}

\noindent so $G_{3}\not\in PW_{1/3}^{\infty}$. By Corollary \ref{ThmUnbound},
the samples $G_{3}(X)$ are unbounded for any separated sequence $X$
with $D^{-}(X)>0$. It is interesting to note that such a function
can still be bounded on a sequence $X$ that is {}``very sparse''
in the sense that $D^{-}(X)=0$. It is easy to check that $G_{3}(3\cdot2^{n})=(-1)^{n}G_{3}(3)$
and $G_{3}(-z)=-G_{3}(z)$, so $G_{3}(X)\in l^{\infty}$ for the sequence
$x_{n}=3\cdot2^{n}\mathrm{sign}(n)$. Some graphs of $G_{3}$ are
shown in Figure \ref{FigG3} below.

\begin{figure}[H]
\subfloat{\includegraphics[scale=0.75]{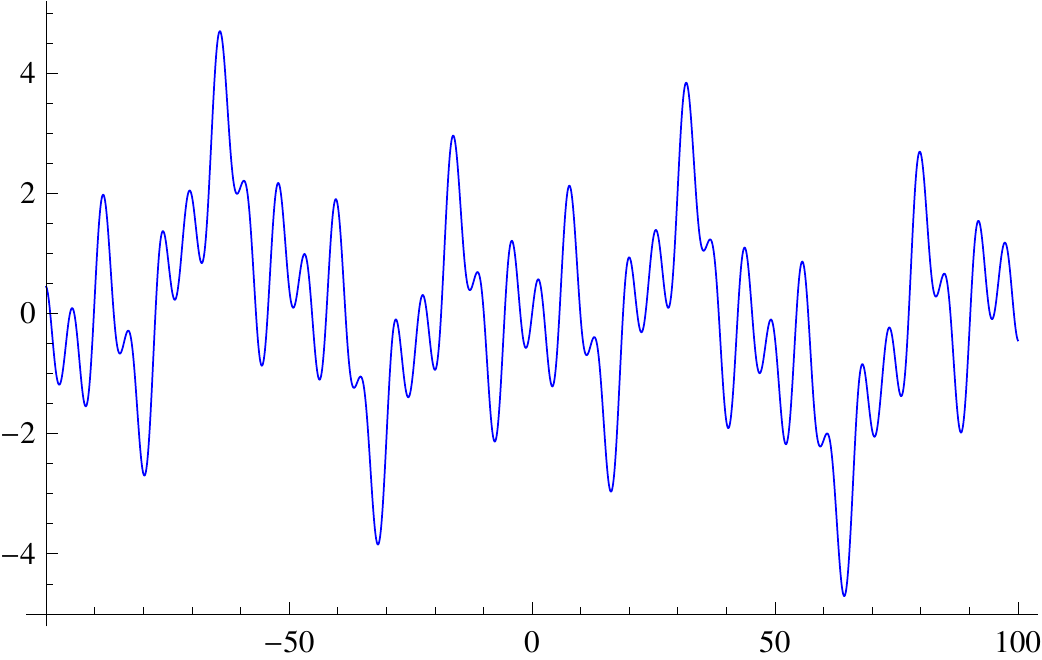}}~~~~~\subfloat{\includegraphics[scale=0.75]{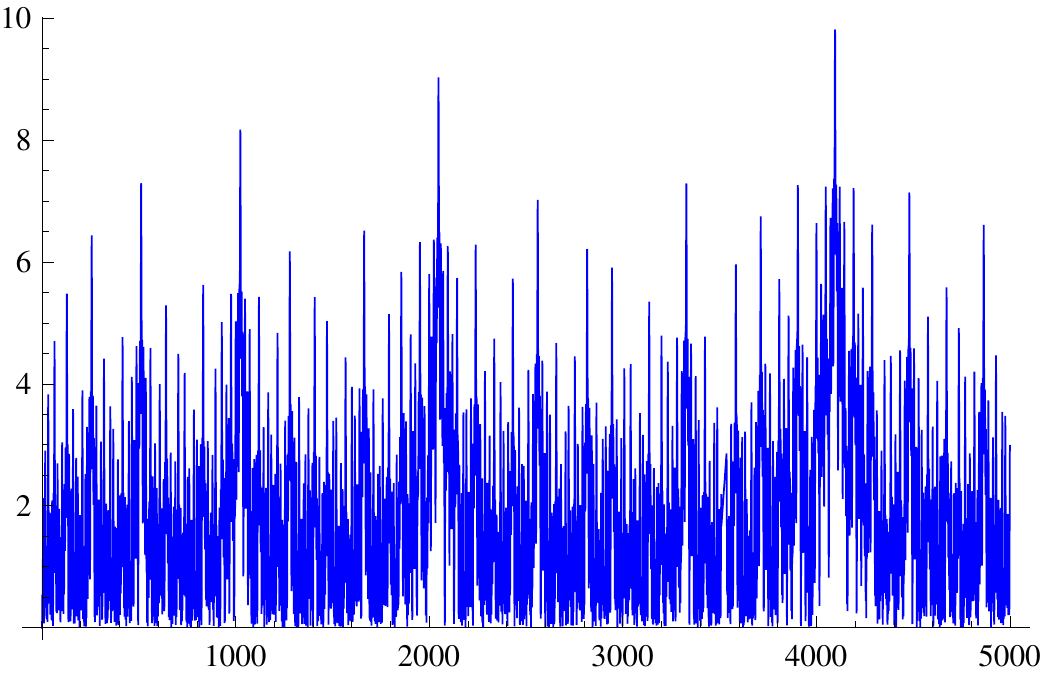}}

\caption{\label{FigG3}Left: The function $G_{3}(z)$ on $[-100,100]$. Right:
The absolute value of $G_{3}(z)$ on $[0,5000]$. The peaks at powers
of $2$ are clearly visible, as well as a self-similarity effect at
different scales.}

\end{figure}

\begin{acknowledgement*}
The author would like to thank Professor Ingrid Daubechies for many
valuable discussions in the course of this work.
\end{acknowledgement*}
\bibliographystyle{plain}
\bibliography{BLI}

\end{document}